\documentclass[reqno]{amsart}
\usepackage{amsfonts,amscd,amsmath,amssymb}
\usepackage{graphicx}
\usepackage{latexsym}
\usepackage{hyperref}
\usepackage[all]{xy}
\usepackage{color}
\usepackage{mathrsfs}
\theoremstyle{plain}

\newtheorem{corollary}{\bf Corollary}

\newtheorem{example}{\bf Example}

\newtheorem{proposition}{\bf Proposition}
\newtheorem{remark}{Remark}

\newtheorem{theorem}{\bf Theorem}
\numberwithin{equation}{section}
\newcommand\tr{\mathrm{tr}}
\newcommand\dv{\mathrm{div}}

\begin{document}

\title[Gradient Ricci almost soliton warped product]{Gradient Ricci almost soliton warped product}

\author{F.E.S. Feitosa$^1$, A.A. Freitas Filho$^2$, J.N.V. Gomes$^3$}

\address{$^{1,2,3}$Departamento de Matem\'atica, Universidade Federal do Amazonas, 69080-900, Manaus, Amazonas, Brazil}
\email{$^{1}$sfeitosa@ufam.edu.br}
\email{$^{2}$aafreitasfilho@ufam.edu.br}
\email{$^{3}$jnvgomes@pq.cnpq.br}
\urladdr{$^{1,2,3}$https://ufam.edu.br}

\author{R.S. Pina$^4$}
\address{$^4$Instituto de Matem\'atica e Estat\'istica, Universidade Federal de Goi\'as, 74001-970, Goi\^ania, Goi\'as, Brazil}
\email{$^4$romildo@ufg.br}
\urladdr{$^4$http://www.mat.ufg.br}



\begin{abstract}
We present the necessary and sufficient conditions for constructing gradient Ricci almost solitons that are realized as warped products. This will be done by means of Bishop-O'Neill's formulas and a particular study of Riemannian manifolds satisfying a Ricci-Hessian type equation. We prove existence results and give an example of particular solutions of the PDEs that arise from our construction. We also prove a rigidity result for a gradient Ricci soliton Riemannian product in the class of gradient Ricci almost soliton warped products under some natural geometric assumptions on the warping function.
\end{abstract}
\maketitle

\section{Introduction}

Warped product manifolds appear in a natural manner in Riemannian geometry and their applications abound. For instance, the Riemannian manifold $\Bbb{S}^n\backslash\Bbb{S}^{n-2}$, i.e., the standard sphere with a codimension two totally geodesic subsphere removed, is isometric to the warped product $\Bbb{S}_+^{n-1}\times_f\Bbb{S}^1$ of an open upper hemisphere and a circle, for some warping function $f\in C^\infty(\Bbb{S}_+^{n-1})$. This was the key ingredient in a beautiful result by Solomon~\cite{Solomon} about harmonic maps from a compact Riemannian manifold into $\Bbb{S}^n$. Specifically, warped products are very useful for studying Einstein manifolds, i.e., a Riemannian manifold $(M,g)$ whose Ricci tensor satisfies $Ric=\lambda g$, for some function $\lambda\in C^\infty(M)$. These two classes of manifolds appear in the context of regular surfaces. Indeed, a surface of revolution is a warped product and any regular surface is an Einstein manifold. Furthermore, they are related in the more general setting which can be confirmed by recent results involving them.

We must emphasize that some generalizations of Einstein manifolds are already much studied in both mathematics and physics frameworks. Among them we would like to highlight the following: Ricci solitons, Ricci almost solitons and $m$-quasi-Einstein manifolds. The first one of them corresponds to self-similar solutions of Ricci flow and often arises as limits of dilations of singularities in the Ricci flow, see Hamilton~\cite{hamilton2}. A special family of the second one of them arises from the Ricci-Bourguignon flow, see Catino et al.~\cite{Catino2} or Catino and Mazzieri~\cite{Catino-Mazzieri}. The third one of them is originated from the study of Einstein warped product manifolds, see Besse~\cite{besse}. More recently, Freitas Filho studied the modified Ricci solitons as a new class of Einstein type manifolds (or metrics) that contains both Ricci solitons and $m$-quasi-Einstein manifolds. This class is closely related to the construction of the Ricci solitons that are realized as warped products. Moreover, a modified Ricci soliton appears as part of a self-similar solution of the modified Harmonic-Ricci flow which results in a new characterization of $m$-quasi-Einstein manifolds, for further details see Freitas Filho~\cite{Airton}.

By analysis of recent results, we can observe that many Einstein type manifolds are closely related to warped products. For example, a classical theorem due to Brinkmann~\cite{Brinkmann} establishes that a Riemannian manifold $M$ is a warped product if and only if there is a nontrivial gradient conformal vector field on $M$. Hence, as it was observed by Petersen and Wylie~\cite{peterneW}, any surface gradient Ricci soliton is immediately a warped product. A locally conformally flat gradient Ricci almost soliton of dimension at least three is, around any regular point of the potential function, locally a warped product with fiber of constant sectional curvature. We observe that this is a consequence of a more general result by Catino~\cite{catino}. An $m$-quasi-Einstein manifold is the base of an Einstein warped product, see Besse~\cite{besse}, Case et al.~\cite{case1} or Kim and Kim~\cite{kim}. In particular, an Einstein warped product manifold is well determined when its base is locally conformally flat, see He et al.~\cite{hepeterwylie}. However, it should be noted that there does not exist a compact Einstein warped product manifold with non-constant warping function if the scalar curvature is non-positive~\cite{kim}. Rimoldi extended this latter result to the case of noncompact bases under some quite natural geometric assumptions on the warping function, see~\cite[Theorem~1]{rimold} for details. In this interesting paper, he used methods from stochastic analysis and $L^p$-Liouville-type theorems to prove scalar curvature estimates and triviality results for an $m$-quasi-Einstein manifold that largely extend previous theorems in~\cite{case1}, see \cite[Sections~3, 4 and 6]{rimold}.

In this setting, Pina and Sousa~\cite{romildo} and Feitosa et al.~\cite{FFG} advanced the study of gradient Ricci solitons that are realized as warped products. They provided a condition for the potential function to depend only on the base as well as the fiber be necessarily Einstein manifold. In particular, in~\cite{romildo} it was presented solutions of PDEs in the case of steady gradient Ricci soliton warped product when the base is conformal to an $n$-dimensional pseudo-Euclidean space invariant under the action of the $(n-1)$-dimensional translation group and the fiber is Ricci-flat. In~\cite{FFG}, it was shown that an expanding or steady gradient Ricci soliton warped product $B^n\times_f \Bbb{F}^m$, $m>1$, whose warping function $f$ reaches both maximum and minimum must be simply a Riemannian product. Moreover, by means of PDEs it was provided the necessary and sufficient conditions for constructing a gradient Ricci soliton warped product. As an application, the authors of~\cite{FFG} constructed a class of expanding Ricci soliton warped product having as fiber an Einstein manifold with non-positive scalar curvature. There have also been discussed some obstructions to this construction, especially when the base of the Ricci soliton warped product is compact. We note that the example in \cite[Corollary~2]{FFG} is rigid in the sense of Petersen and Wylie~\cite{PW-2009}.

In this paper, we study gradient Ricci almost solitons that are realized as warped products. Ricci almost solitons were introduced by Pigola et al.~\cite{prrs} where essentially the authors modified the definition of Ricci solitons by adding the condition on the parameter $\lambda$ to be a nonconstant function. They provided existence and rigidity results, investigated some topological properties and derived a number of differential identities involving relevant geometric quantities. Moreover, some basic tools from the weighted manifold theory such as general weighted volume comparisons and maximum principles at infinity for diffusion operators were also discussed.
The structure equations for Ricci almost solitons were published shortly after by Barros and Ribeiro~\cite{br2}. As a consequence of these equations they showed that a compact nontrivial Ricci almost soliton is isometric to a sphere provided it is either of constant scalar curvature or its associated vector field is conformal.

It was later observed that a special solution $g(t)(x)=\tau(t,x)\varphi_t ^*g_0(x)$ of the Ricci flow $\partial g(t)/\partial t=-2Ric_{g(t)}$ on a Riemannian manifold $(M,g_0)$ generates a Ricci almost soliton, where $\varphi_t$ are diffeomorphisms of $M$, with $\varphi_0=id_M$, and for each $t$, $\tau(t,x)$ is a positive smooth function on $M$, with $\tau(0,x)=1$, see Gomes~\cite{Gomes} or Sharma~\cite{Sharma}. But, unlike Ricci solitons, it is not yet possible to characterize a Ricci almost soliton, in its generality, by some geometric flow. However, as aforementioned, a special family of Ricci almost solitons (namely, the $\rho$-Einstein solitons) arises from the Ricci-Bourguignon flow. Indeed, it is proved in \cite{Catino2} that the self-similar solutions of the Bourguignon flow~\cite{Bourguignon} satisfy the almost soliton equation $Ric+\nabla^2\psi = (\rho R + \lambda) g$, where $\rho,\lambda\in \Bbb{R}$ are fixed constants and $R$ is the scalar curvature of the metric $g$. Here, we highlight the paper of Catino and Mazzieri~\cite{Catino-Mazzieri} from which we can know the gradient $\rho$-Einstein solitons and several structural results for shrinkers of the Ricci-Bourguignon flow.

Let us now be a little bit more precise, we say that a complete Riemannian manifold $(M^k,g)$ is a Ricci almost soliton, if there exist a vector field $X$ satisfying
\begin{equation}\label{eqfund1}
Ric +\frac{1}{2}\mathscr{L}_{X}\rm{g}=\lambda \rm{g},
\end{equation}
for some soliton function $\lambda\in C^\infty(M)$, where $Ric$ and $\mathscr{L}$ stand for the Ricci tensor and the Lie derivative, respectively. We shall refer to this equation as the fundamental equation of a Ricci almost soliton $\big(M^k,g, X,\lambda\big)$. When the vector field $X$ is the gradient of a smooth function $\psi\in C^\infty(M)$ the manifold will be called a gradient Ricci almost soliton. Accordingly equation~\eqref{eqfund1} becomes
\begin{equation}\label{eqfund2}
{Ric+\nabla^{2}\psi=\lambda g},
\end{equation}
where $\nabla^{2}\psi$ stands for the Hessian of the potential function $\psi$. Taking the trace of equation~\eqref{eqfund2} we obtain
\begin{equation*}
R+\Delta\psi = k\lambda.
\end{equation*}

The next identity is just one of the structure equations that appear in~\cite{br2}
\begin{equation*}
d(R+|\nabla\psi|^{2}-2(k-1)\lambda)=2\lambda d\psi.
\end{equation*}

It follows from these two relations that
\begin{equation}\label{GenHam-C}
-2\lambda d\psi + d\big((2-k)\lambda+|\nabla\psi|^{2} - \Delta\psi\big)=0.
\end{equation}

In this direction we deduce similar equation to \eqref{GenHam-C} for the base of a gradient Ricci almost soliton warped product, see equation~\eqref{EQMthm2} below.

Recall that the warped product $M=B^n\times_f\Bbb{F}^m$ of two Riemannian manifolds is simply their Riemannian product endowed with the metric
\begin{equation*}
g=\pi^{*}g_{B}+(f\circ \pi)^{2}\sigma^{*}g_{\Bbb{F}},
\end{equation*}
where $\pi:M\rightarrow B$ and $\sigma:M\to\Bbb{F}$ are canonical projection maps, the positive smooth function $f$ on $B$ is the warping function. Besides the notation  $\tilde{\varphi}=\varphi\circ\pi$ stands for the lift of a smooth function $\varphi$ on $B$ to $M$.

Again by Brinkmann's theorem we can affirm that any surface gradient Ricci almost soliton is a warped product. Moreover, Ricci almost solitons that are realized as Einstein warped products, with unidimensional base and Einstein fiber, were constructed in~\cite{prrs}. By using Lemma~$1.1$ of~\cite{prrs}, we can prove that the warped product $M=\Bbb{R}\times_f\Bbb{H}^m$ with metric $g=dt^2+f^2g_0$, has a structure of Ricci almost soliton $\big(M,g,\nabla\tilde \varphi,\tilde\lambda\big)$, where $g_0$ is the standard metric of $\Bbb{H}^m$ and the functions involved are the respective lifts of $\varphi(t)=\sinh(t)$ and $\lambda (t)=\sinh(t)-m$ whereas the warping function is $f(t)=\cosh(t)$. A straightforward computation shows that these functions satisfy equations~\eqref{EQMthm} and \eqref{EQMthm2} below, with $\mu=-(m-1)$ given by~\eqref{CMthm}. Hence, the referred example can also be obtained by our Theorem~\ref{CARSPW}, see Example~\ref{ExRigoli}. We emphasize that the construction in~\cite{prrs} is only applicable to Ricci almost soliton realized as Einstein warped product. Thus, we must be careful when using it, since by Bonnet-Myers theorem a $k$-dimensional complete Riemannian manifold $M^k$ is compact provided that its Ricci curvature is bounded below by $(k-1)c>0$.

Our purpose is to obtain the necessary and sufficient conditions to construct a gradient Ricci almost soliton (non-necessarily Einstein) warped product. We recall that Bryant constructed a steady Ricci soliton as the warped product $(0,+\infty) \times_f \Bbb{S}^{m}$, $m>1$, with a radial warping function $f$. Bryant did not himself publish this result, but it can be checked in \cite{Chow}. Since a warped product is a complete manifold for all warping function if and only if both base and fiber are complete manifolds (see Bishop and O'Neill~\cite{BO} or O'Neill~\cite{oneill}), the difficulty in the construction of Bryant was to show the completeness of his example. In contrast to the Bryant's construction and according to the result of Bishop and O'Neill, we consider both the base and the fiber to be complete. So, our concern will be the solutions of the PDEs that arise from our construction, see equations~\eqref{EQMthm} and \eqref{EQMthm2}. However, Bryant's example still remains a prototype for us.

The necessary conditions for our construction are described in the following two propositions.

\begin{proposition}\label{PP5}
Let $(B^{n},g)$ be a Riemannian manifold with three smooth functions $f>0$, $\lambda$ and $\varphi$ satisfying
\begin{equation}\label{EQMthm}
Ric+\nabla^{2}\varphi=\lambda g+\frac{m}{f}\nabla^{2}f
\end{equation}
and
\begin{equation}\label{EQMthm2}
-2\lambda d\varphi + d\big((2-m-n)\lambda + |\nabla\varphi|^{2} - \Delta\varphi - \frac{m}{f}\nabla\varphi(f)\big)=0,
\end{equation}
for some constant $m\in\Bbb{R}$, with $m\neq 0$. Then $f$, $\lambda$  and $\varphi$ satisfy
\begin{equation}\label{CMthm}
\lambda f^{2}+f\Delta f +(m-1)|\nabla f|^{2} -f\nabla\varphi(f)=\mu,
\end{equation}
for a constant $\mu\in\Bbb{R}.$
\end{proposition}

Throughout this paper, the functions $\tilde\varphi$ and $\tilde\lambda$ stand for the lifts to warped product of the smooth functions $\varphi$ and $\lambda$ defined on the base of the warped product, respectively. By Proposition~\ref{Prop2}, we know when $(B^n\times_f\Bbb{F}^m,g,\nabla\psi,\tilde\lambda)$ has $\psi=\tilde\varphi$ as potential function, for some smooth function $\varphi$ on $B$.

\begin{proposition}\label{PP2}
Let $M=B^n\times_f\Bbb{F}^m$, with $m>1$, be a warped product and $\varphi$, $\lambda$ be smooth functions on $B$ so that $(M,g,\nabla\tilde\varphi,\tilde\lambda)$ is a gradient Ricci almost soliton. Then we have
\begin{equation*}
Ric_B+\nabla^2\varphi=\lambda g_B+\frac{m}{f}\nabla^2f
\end{equation*}
and $^{\Bbb{F}}\!Ric=\mu g_{\Bbb{F}}$ with $\mu$ satisfying
\begin{equation*}
\mu=\lambda f^{2}+f\Delta f+(m-1)|\nabla f|^{2}-f\nabla\varphi(f).
\end{equation*}
\end{proposition}

Further generalizations of Einstein metrics have been considered in~\cite{master}, where equation~\eqref{eqfund2} is replaced by what the author calls the {\it Ricci-Hessian equation}, namely,
\begin{equation}\label{eqmaster}
Ric+\alpha\nabla^{2}\psi=\gamma g,
\end{equation}
where $\alpha$ and $\gamma$ are smooth functions. Notice that since the author is interested in conformal changes of K\"ahler-Ricci solitons which give rise to new K\"ahler metrics, the presence of the function $\alpha$ is vital in his investigation.

We point out that Riemannian manifolds satisfying a Ricci-Hessian type equation~\eqref{EQMthm} by itself is already quite interesting. By Proposition~\ref{PP2}, the base of every gradient Ricci almost soliton that is realized as warped product has a structure given by a Ricci-Hessian type equation. In Remark~\ref{Remark-RHTM}, we give some conditions for \eqref{EQMthm} becomes a Ricci-Hessian equation. Example~$1$ in~\cite{FFG} shows that the standard sphere and the hyperbolic space both satisfy a Ricci-Hessian type equation. Moreover, some geometric properties of equation~\eqref{EQMthm} can be find in the preprint by Gomes and Matos Neto~\cite{Gomes-Manoel}.

By using the Bishop-O'Neill's formulas~\cite{BO,oneill} we construct a gradient Ricci almost soliton warped product as follows.

\begin{theorem}\label{CARSPW}
Let $(B^n,g_B)$ be a complete Riemannian manifold with three smooth functions $f>0$, $\lambda$ and $\varphi$ satisfying \eqref{EQMthm} and \eqref{EQMthm2}. Take the constant $\mu$ satisfying \eqref{CMthm} and a complete Riemannian manifold $(\Bbb{F}^m,g_{\Bbb{F}})$ with Ricci tensor $^{\Bbb{F}}\!Ric=\mu g_{\Bbb{F}}$ and $m>1$. Then $(B^n\times_f\Bbb{F}^m,g,\nabla\tilde\varphi,\tilde\lambda)$ is a gradient Ricci almost soliton warped product.
\end{theorem}

Notice that Theorem~\ref{CARSPW} describes the necessary and sufficient conditions for constructing gradient Ricci almost solitons that are realized as warped products. As an application, we solve a particular case of PDEs in \eqref{EQMthm} and \eqref{EQMthm2} in order to construct an explicit new nontrivial almost soliton structure on $\Bbb{R}^n\times_f\Bbb{F}^m$, where $\Bbb{F}^m$ is any Ricci flat Riemannian manifold.

\begin{corollary}\label{CorCARSWP}
Let $\Bbb{R}^n$ be an Euclidean space with coordinates $x=(x_1,\ldots,x_n)$ and metric $g_{ij}=e^{2\xi}\delta_{ij}$, where $\xi=\sum_{i=1}^n \alpha_ix_i$, $\alpha_i\in\Bbb{R}$ and $n\geq 3$. The functions
\begin{equation*}
f=e^{\xi}, \quad \varphi=\frac{c_{1}}{2}e^{2\xi}-\dfrac{(2-m-n)}{2}\xi+c_{2} \quad \mbox{and} \quad \lambda=c_{1}+\frac{(2-m-n)}{2}e^{-2\xi}
\end{equation*}
satisfy equations~\eqref{EQMthm} and \eqref{EQMthm2} on $\Bbb{R}^n$, for some constants $c_1$ and $c_2$. Moreover, the constant $\mu$ given by \eqref{CMthm} is null. Then, for any complete Ricci flat Riemannian manifold $\Bbb{F}^m$, with $m>1$, the quadruple $(\Bbb{R}^n\times_f\Bbb{F}^m,g, \nabla\tilde\varphi,\tilde\lambda)$ is a gradient Ricci almost soliton warped product.
\end{corollary}

Using the strong maximum principle for a specific elliptic operator, we prove a rigidity result (in the spirit of~\cite{kim}) for a gradient Ricci soliton Riemannian product in the class of gradient Ricci almost soliton warped products under some natural geometric assumptions on the warping function as follows.

\begin{theorem}\label{thmTrivial}
Let $(M=B^n\times_f\Bbb{F}^m,g,\nabla\tilde\varphi,\tilde\lambda)$, with $m>1$, be a gradient Ricci almost soliton warped product. Then $M$ is simply a gradient Ricci soliton Riemannian product if either of these conditions is satisfied:
\begin{enumerate}
\item\label{thmTrivial-item1} $f$ reaches a minimum and $\lambda\geq\frac{\mu}{f^2}$  (or $f$ reaches a maximum and $\lambda\leq\frac{\mu}{f^2}$).
\item\label{thmTrivial-item2} $\lambda\leq0$ and $\lambda(p)\leq\lambda(q)$, where $p$ and $q$ are maximum and minimum points of $f$, respectively.
\end{enumerate}
\end{theorem}

Very recently, in the interesting preprint~\cite{TB}, Tenenblat and Borges extended our characterizations of Ricci almost soliton warped products to the semi-Riemannian setting and by allowing the potential function to depend on the fiber. We point out that our results are previous to those of the cited authors, as proved in our $2015$ arXiv version.

It is also worth mentioning the work of Calvi\~no-Louzao et al.~\cite{EMER} in which they showed that a locally homogeneous proper Ricci almost soliton is either of constant sectional curvature or is locally isometric to a product $\Bbb{R}\times\Bbb{N}(c)$, where $\Bbb{N}(c)$ is a space of constant curvature. For more results about Ricci almost solitons, we refer the reader to~\cite{CHG,bgr,br2,Catino-Mazzieri,prrs}.

\section{Existence conditions for the Ricci almost soliton warped product}\label{Preliminaries}

In this section, we shall follow both the notation and the terminology of O'Neill \cite{oneill}. Given a warped product $M=B^n\times_{f}\Bbb{F}^m$, the manifold $B$ is called the \textit{base} of $M$ and $\Bbb{F}$ - the \textit{fiber}. The set of all horizontal lifts $\tilde{X}$ is denoted by $\mathfrak{L}(B)$, whereas the set of all vertical lifts $\tilde{V}$ is denoted by $\mathfrak{L}(\Bbb{F})$. From now on, if $X\in\mathfrak{X}(B)$, when there is no danger of confusion, we will use the same notation for its horizontal lift. We shall follow similar convention for the vertical lift of $V\in\mathfrak{X}(\Bbb{F})$.

Tangent vectors to the leaves are \textit{horizontal} and tangent vectors to the fibers are \textit{vertical}. We denote by $\mathcal{H}$ the orthogonal projection of $T_{(p,q)}M$ onto its horizontal subspace $T_{(p,q)}(B\times q)$ and by $\mathcal{V}$ the projection onto the vertical subspace $T_{(p,q)}(p\times\Bbb{F})$. It is well known that the gradient of the lift $h\circ\pi$ of a smooth function $h$ on $B$ to $M$ is the lift  of the gradient of $h$, see~\cite[Lemma~$34$]{oneill}. Thus, there should be no confusion if we simplify write $\tilde h$ for $h\circ\pi$. This way,  the gradient, the Hessian and Laplacian of $\tilde h$ calculated in the metric of $M$ will be respectively denoted  by $\nabla\tilde h$, $\nabla^2\tilde h$ and $\Delta\tilde h$, where $\Delta=\tr(\nabla^{2})$. We will denote by $D$, $\nabla$ and $^{\Bbb{F}}\!\nabla$ the Levi-Civita connections of the $M$, $B$ and $\Bbb{F}$, respectively. Moreover, we shall write $Ric$ for the Ricci tensor of the warped product, $^{B}\!Ric$ for the lift of the Ricci tensor of $B$ and $^{\Bbb{F}}\!Ric$ for the lift of the Ricci tensor of $\Bbb{F}$. Finally, we denote by $H^{h}$ the lift  of the Hessian $\nabla^2h$. Observe that for all $Y,Z\in\mathfrak{L}(B)$ we have $\nabla^2\tilde h(Y,Z)=H^h(Y,Z)$. We are now ready to prove our first result.
\begin{proposition} \label{Prop2}
Let $M=B^n\times_f\Bbb{F}^m$, with $m>1$, be a warped product and  $\psi\in C^{\infty}(M)$, $\lambda\in C^{\infty}(B)$ be two smooth functions such that $(M,g,\nabla\psi,\tilde\lambda)$ is a gradient Ricci almost soliton, with $\mathcal{H}(\nabla\psi)\in\mathfrak{L}(B)$ and $\mathcal{V}(\nabla\psi)\in\mathfrak{L}(F)$. Then $\psi=\tilde{\varphi}$ for some function $\varphi\in C^{\infty}(B)$ and the equation~\eqref{EQMthm2} is valid for the functions $f,\,\varphi$ and $\lambda$.
\end{proposition}
\begin{proof}
By Corollary~$43$ of \cite{oneill}, for all $Y\in\mathfrak{L}(B)$ and $V\in\mathfrak{L}(\Bbb{F})$ is valid
\begin{equation}\label{ii,Cor43,O'Neill}
Ric(Y,V)=0.
\end{equation}
Since $\nabla\psi=\mathcal{H}(\nabla\psi)+\mathcal{V}(\nabla\psi)$, by the fundamental equation, we have
\begin{equation*}
0=Ric(Y,V)-\tilde\lambda g(Y,V)=-g(D_{Y}\mathcal{H}(\nabla\psi),V)-g(D_{Y}\mathcal{V}(\nabla\psi),V).
\end{equation*}
Proposition~$35$ of \cite{oneill} ensures that
\begin{equation} \label{ii,Pro35,O'Neill}
D_{Y}\mathcal{H}(\nabla\psi)\in\mathfrak{L}(B) \quad \mbox{and} \quad D_{Y}V=D_{V}Y=\frac{Y(f)}{f}V.
\end{equation}
Thus,
\begin{equation*}
0=g(D_{Y}\mathcal{V}(\nabla\psi),V)=Y(\ln f)g\big(\mathcal{V}(\nabla\psi),V\big).
\end{equation*}
From where we conclude that $\nabla\psi\in\mathfrak{L}(B)$. By uniqueness of the lift we have that $\psi=\tilde\varphi$ for some smooth function $\varphi$ on $B$. This proves the first assertion of the proposition. From \eqref{GenHam-C} we have
\begin{equation} \label{L2-1}
-2\tilde\lambda d\tilde{\varphi}+(2-m-n)d\tilde\lambda+d(|\nabla\tilde{\varphi}|^{2})-d(\Delta\tilde{\varphi})=0.
\end{equation}
By Lemma~$34$ of \cite{oneill} we have
\begin{equation}\label{L2-2}
\nabla\tilde{\varphi}=\widetilde{\nabla\varphi}.
\end{equation}
Using Proposition~$35$ of \cite{oneill} it is possible show that
\begin{equation} \label{L2-2.1}
\Delta\tilde{\varphi}=\Delta\varphi+\frac{m}{f}\nabla\varphi(f).
\end{equation}
Plugging \eqref{L2-2} and \eqref{L2-2.1} in \eqref{L2-1} we obtain equation~\eqref{EQMthm2}.
\end{proof}

Now, we proof the second proposition which gives the necessary conditions for a warped product to admit a structure of gradient Ricci almost soliton.

\subsection{Proof of Proposition~\ref{PP2}}\label{Proof-PP2}
\begin{proof}
Corollary~$43$ of \cite{oneill} ensures that
\begin{equation}\label{i,Cor43,O'Neill}
Ric(Y,Z)=\!^{B}\!Ric(Y,Z)-\frac{m}{f}H^{f}(Y,Z),
\end{equation}
for all $Y,Z\in\mathfrak{L}(B)$. Using $\nabla^{2}\tilde{\varphi}(Y,Z)=H^{\varphi}(Y,Z)$ along with the fundamental equation we can write
\begin{equation*}
^{B}\!Ric(Y,Z)=\tilde\lambda g(Y,Z)-H^{\varphi}(Y,Z)+\frac{m}{f}H^{f}(Y,Z).
\end{equation*}
Follows that
\begin{equation*}
Ric_B=\lambda g_B-\nabla^2\varphi+\frac{m}{f}\nabla^2f
\end{equation*}
on the base $B$. This proves the first assertion of the proposition. Again, by Corollary~$43$ of \cite{oneill} and the fundamental equation, we have
\begin{equation}\label{iii,Cor43,O'Neill}
^{\Bbb{F}}\!Ric(V,W)=\tilde\lambda g(V,W)-\nabla^{2}\tilde{\varphi}(V,W)+\Big(\frac{\Delta f}{f}+(m-1)\frac{|\nabla f|^{2}}{f^{2}}\Big)g(V,W)
\end{equation}
for all $V,W\in\mathfrak{L}(V)$. Since $\nabla\tilde\varphi\in\mathfrak{L}(B)$ from \eqref{ii,Pro35,O'Neill} we get
\begin{equation}\label{eqAux1Thm2}
\nabla^{2}\tilde{\varphi}(V,W) = g(D_{V}\nabla\tilde\varphi,W)= g\Big(\frac{\nabla\widetilde\varphi(f)}{f}V,W\Big) = f\nabla\varphi(f)g_{\Bbb{F}}(V,W).
\end{equation}
Thus, equation~\eqref{iii,Cor43,O'Neill} can be written as follows
\begin{equation*}
^{\Bbb{F}}\!Ric(V,W)=\big(\lambda f^{2}+f\Delta f+(m-1)|\nabla f|^{2}-f\nabla\varphi(f)\big)g_{\Bbb{F}}(V,W),
\end{equation*}
which is sufficient to complete the proof of the proposition.
\end{proof}
\begin{remark}\label{Remark-RHTM}
Notice that equation~\eqref{EQMthm} is a Ricci-Hessian equation provided
\begin{equation*}
\nabla^2\varphi-\frac{m}{f}\nabla^2f=\alpha\nabla^2\gamma,
\end{equation*}
for some smooth functions $\alpha$ and $\gamma$. For instance, we suppose that there exists $c\in\Bbb{R}$ such that $\varphi>-\frac{mc^2}{2}$, so that this latter equation
is satisfied by the functions
\begin{equation*}
\gamma=-mc+\sqrt{m^2c^2+2m\varphi},\quad \alpha=\frac{\gamma}{m}+c+1\quad \mbox{and}\quad f=e^{-\frac{\gamma}{m}}.
\end{equation*}
Indeed, we can verify easily that
\begin{equation*}
-\frac{m}{f}\nabla^2f=\nabla^2\gamma-\frac{1}{m}d\gamma\otimes d\gamma,
\end{equation*}
and that $\gamma$ is a solution to the equation $\gamma^2+2mc\gamma-2m\varphi=0$. Thus,
\begin{equation*}
\nabla^2\varphi=\left(\frac{\gamma}{m}+c\right)\nabla^2\gamma+\frac{1}{m}d\gamma\otimes d\gamma,
\end{equation*}
which is sufficient to obtain the required equation.
\end{remark}

In what follows, we consider gradient Ricci almost soliton warped products, where the bases are conformal to an Euclidean spaces $\Bbb{R}^n$ which are invariant under the action of the $(n-1)$-dimensional translation group. For this purpose, it will suffice to construct the solutions of the equation~\eqref{EQMthm} of the form $f(\xi)>0$, $\varphi(\xi)$ and $\lambda(\xi)$, that is, they only depend on $\xi=\sum_{i=1}^n \alpha_ix_i$, $\alpha_i\in\Bbb{R}$. Whenever $\sum_{i=1}^n\alpha_i^2\neq 0$, without loss of generality, we may consider $\sum_{i=1}^n\alpha_i^2=1$. The following proposition provides the system of ordinary differential equations that must be satisfied by such solutions.

\begin{proposition}\label{ExRomi}
Let $\Bbb{R}^n$, with $n\geq 3$, be an Euclidean space with coordinates $x=(x_1,\ldots,x_n)$ and metric $g_{ij}=\frac{1}{F(\xi)^2}\delta_{ij}$, where $F(\xi)\in C^\infty(\Bbb{R}^n)$,  $\xi=\sum_{i=1}^n\alpha_ix_i$, $\alpha_i\in\Bbb{R}$. For every positive smooth functions $F(\xi)$ and $f(\xi)$ we get smooth functions $\varphi(\xi)$ and $\lambda(\xi)$ satisfying \eqref{EQMthm} by
\begin{equation}\label{EQMthmInv1}
(n-2)\frac{F''}{F}+\varphi''+2\frac{F'}{F}\varphi'=\frac{m}{f}\Big(f''+2\frac{F'}{F}f'\Big)
\end{equation}
and
\begin{equation}\label{EQMthmInv2}
\frac{F''}{F}-(n-1)\Big(\frac{F'}{F}\Big)^{2}-\frac{F'}{F}\varphi'=\frac{\lambda}{F^{2}}-m\dfrac{f'}{f}\frac{F'}{F}.
\end{equation}
\end{proposition}
\begin{proof}
Since the metric $g$ is conformal the canonical metric $g_0$ of $\Bbb{R}^n$, it is well known that
\begin{equation*}
Ric_{g}=(n-2)\frac{1}{F}\nabla^{2}F + \frac{1}{F^{2}}\big(F\Delta F-(n-1)|\nabla F|^{2}\big)g_{0},
\end{equation*}
where the two summands appearing in the second term of this equation are calculated in the metric $g_0$. Moreover, for every $f\in C^\infty(\Bbb{R}^n)$ the following are valid
\begin{eqnarray*}
(\nabla^{2}f)_{ij}&=&f_{x_{i}x_{j}}+\frac{F_{x_{j}}}{F}f_{x_{i}}+\frac{F_{x_{i}}}{F}f_{x_{j}} \quad  \mbox{for} \quad i\neq j;\\
(\nabla^{2}f)_{ii}&=&f_{x_{i}x_{i}}+2\frac{F_{x_{i}}}{F}f_{x_{i}}-\frac{1}{F}\sum_k F_{x_{k}}f_{x_{k}}.
\end{eqnarray*}
So, we need to analyze equation \eqref{EQMthm} in two cases. For $i\neq j$, it rewrites as
\begin{equation} \label{I1}
(n-2)\frac{F_{x_{i}x_{j}}}{F}+\varphi_{x_{i}x_{j}}+\frac{F_{x_{j}}}{F}\varphi_{x_{i}}+\frac{F_{x_{i}}}{F}\varphi_{x_{j}}=\frac{m}{f}\Big(f_{x_{i}x_{j}}
+\frac{F_{x_{j}}}{F}f_{x_{i}}+\frac{F_{x_{i}}}{F}f_{x_{j}}\Big)
\end{equation}
and for $i=j$,
\begin{eqnarray}\label{I2}
\nonumber&&(n-2)\frac{F_{x_{i}x_{i}}}{F}+\frac{1}{F}\sum_i F_{x_{i}x_{i}}-\frac{(n-1)}{F^2}\sum_iF_{x_{i}}^{2}
+\varphi_{x_{i}x_{i}}+2\frac{F_{x_{i}}}{F}\varphi_{x_{i}}\\
&=& \frac{1}{F}\sum_k F_{x_{k}}\varphi_{x_{k}} +\frac{\lambda}{F^{2}}+\frac{m}{f}\Big(f_{x_{i}x_{i}}+2\frac{F_{x_{i}}}{F}f_{x_{i}}-\frac{1}{F}\sum_k F_{x_{k}}f_{x_{k}}\Big).
\end{eqnarray}
We now assume that the argument $\xi$ of the functions  $F(\xi)$, $f(\xi)$ and $\varphi(\xi)$ is of the form $\xi=\sum_{i=1}^n\alpha_{i}x_{i}$. Hence,
we have $F_{x_{i}}=F'\alpha_{i}$ and $F_{x_{i}x_{j}}=F''\alpha_{i}\alpha_{j}$ where the superscript $'$ denotes the derivative with respect to $\xi$. Using the same reasoning  for $f$ and $\varphi$,  we rewrite the equations \eqref{I1} and \eqref{I2} as
\begin{equation}\label{I1.1}
(n-2)\frac{F''}{F}\alpha_{i}\alpha_{j}+\varphi''\alpha_{i}\alpha_{j}+2\frac{F'}{F}\varphi'\alpha_{i}\alpha_{j}=\frac{m}{f}\Big(f''\alpha_{i}\alpha_{j}
+2\frac{F'}{F}f'\alpha_{i}\alpha_{j}\Big)
\end{equation}
and
\begin{eqnarray}\label{I2.1}
\nonumber&&(n-2)\frac{F''}{F}\alpha_{i}^{2}+\frac{F''}{F}\sum_i\alpha_{i}^{2}-(n-1)\Big(\frac{F'}{F}\Big)^{2}\sum_i\alpha_{i}^{2}+\varphi''\alpha_{i}^{2}
+2\frac{F'}{F}\varphi'\alpha_{i}^{2}\\
&=&\frac{F'}{F}\varphi'\sum_k \alpha_{k}^{2}+\frac{\lambda}{F^{2}}+\dfrac{m}{f}\Big(f''\alpha_{i}^{2}+2\frac{F'}{F}f'\alpha_{i}^{2}
-\frac{F'}{F}f'\sum_k \alpha_{k}^{2}\Big).
\end{eqnarray}
Since $n\geq 3$, we can choose this invariance so that at least two indices $i,j$ are such that $\alpha_{i}\alpha_{j}\neq 0$ and $\sum_{i=1}^n\alpha_{i}^{2}=1$. Thus, \eqref{I1.1} and \eqref{I2.1} are summarized, respectively, by
\begin{equation}\label{I1.2}
(n-2)\frac{F''}{F}+\varphi''+2\frac{F'}{F}\varphi'=\frac{m}{f}\Big(f''+2\frac{F'}{F}f'\Big)
\end{equation}
and
\begin{eqnarray}\label{I2.2}
&&(n-2)\frac{F''}{F}\alpha_{i}^{2}+\frac{F''}{F}-(n-1)\Big(\frac{F'}{F}\Big)^{2}+\varphi''\alpha_{i}^{2}+2\frac{F'}{F}\varphi'\alpha_{i}^{2}\\
\nonumber&=& \frac{F'}{F}\varphi' +\frac{\lambda}{F^{2}}+\dfrac{m}{f}\Big(f''\alpha_{i}^{2}+2\frac{F'}{F}f'\alpha_{i}^{2}-\frac{F'}{F}f'\Big).
\end{eqnarray}
Plugging, \eqref{I1.2} in \eqref{I2.2}, we get
\begin{equation}\label{I5}
\frac{F''}{F}-(n-1)\Big(\frac{F'}{F}\Big)^{2}-\frac{F'}{F}\varphi'=\frac{\lambda}{F^{2}}-m\dfrac{f'}{f}\frac{F'}{F}.
\end{equation}
This concludes the proof of the proposition.
\end{proof}

\section{Proof of the main results}\label{Proof-Results}

We are in the right position to prove our main results. First of all, we recall that each $(0,2)$-tensor $T$ on $(M,g)$ can be associated to a unique $(1,1)$-tensor by
$g(T(Z),Y) := T(Z,Y)$ for all $Y,Z\in\mathfrak{X}(M)$. We shall slightly abuse notation here and will also write $T$ for this $(1,1)$-tensor. So, we consider the $(0,1)$-tensor given by
\begin{equation*}
(\dv T)(v)(p) = \mathrm{tr}\big(w \mapsto (\nabla_w T)(v)(p)\big),
\end{equation*}
where $p\in M$ and $v,w\in T_pM.$ Thus, we get
\begin{equation*}
\dv(\varphi T)=\varphi \dv T+ T(\nabla\varphi,\cdot) \quad \mbox{and} \quad \nabla(\varphi T)=\varphi\nabla T+d\varphi\otimes T
\end{equation*}
for all $\varphi\in C^\infty(M).$ In particular, we have $\dv(\varphi g)=d\varphi$. Moreover, the following two general facts are well known in the literature
\begin{equation*}
 \dv\nabla^2\varphi = Ric(\nabla\varphi,\cdot)+ d\Delta\varphi \quad \mbox{and} \quad \frac{1}{2}d|\nabla\varphi|^2 = \nabla^2\varphi(\nabla\varphi,\cdot).
\end{equation*}

These identities will be used in what follows without further comments.

\subsection{Proof of Proposition~\ref{PP5}}
\begin{proof} From \eqref{EQMthm} we obtain
\begin{eqnarray*}
S=n\lambda+\frac{m}{f}\Delta f-\Delta\varphi,
\end{eqnarray*}
where $S$ is the scalar curvature of $B$. Thus,
\begin{eqnarray} \label{Bianc1}
dS=nd\lambda-\frac{m}{f^{2}}\Delta f df+\frac{m}{f}d(\Delta f)-d(\Delta\varphi)
\end{eqnarray}
Let us now use the second contracted Bianch identity, namely
\begin{equation}\label{Bianc}
0=-\frac{1}{2}dS+\dv Ric.
\end{equation} We compute
\begin{eqnarray*}
\dv Ric&=&d\lambda+m \dv\Big(\frac{1}{f}\nabla^{2}f\Big)-\dv(\nabla^{2}\varphi)\\
&=&d\lambda+m\Big(\frac{1}{f}\dv(\nabla^{2}f)-\frac{1}{f^{2}}(\nabla^{2}f)(\nabla f,\cdot)\Big)-\dv(\nabla^{2}\varphi)\\
&=&d\lambda+\dfrac{m}{f}Ric(\nabla f,\cdot)+\frac{m}{f}d(\Delta f)-\frac{m}{2f^{2}}d(|\nabla f|^{2})-Ric(\nabla\varphi,\cdot)-d(\Delta\varphi).
\end{eqnarray*}
From \eqref{EQMthm} we have
\begin{equation*}
Ric(\nabla f,\cdot)=\lambda df+\dfrac{m}{2f}d(|\nabla f|^{2})-(\nabla^{2}\varphi)(\nabla f,\cdot)
\end{equation*}
and
\begin{equation*}
Ric(\nabla \varphi,\cdot)=\lambda d\varphi+\dfrac{m}{f}(\nabla^{2}f)(\nabla\varphi,\cdot)-\frac{1}{2}d(|\nabla\varphi|^{2}).
\end{equation*}
This way
\begin{eqnarray*}
\dv Ric&=&d\lambda+\dfrac{m}{f}\Big(\lambda df+\dfrac{m}{2f}d(|\nabla f|^{2})-(\nabla^{2}\varphi)(\nabla f,\cdot)\Big)+\frac{m}{f}d(\Delta f)\\
&&-\frac{m}{2f^{2}}d(|\nabla f|^{2})-\Big(\lambda d\varphi+\dfrac{m}{f}(\nabla^{2}f)(\nabla\varphi,\cdot)-\frac{1}{2}d(|\nabla\varphi|^{2})\Big)-d(\Delta\varphi).
\end{eqnarray*}
Since $d(\nabla\varphi(f))=(\nabla^{2}\varphi)(\nabla f,\cdot)+(\nabla^{2}f)(\nabla\varphi,\cdot)$, then
\begin{eqnarray}\label{Bianc2}
\nonumber\dv Ric &=&d\lambda+\dfrac{m}{f}\lambda df+\dfrac{m(m-1)}{2f^{2}}d(|\nabla f|^{2})-\dfrac{m}{f}d(\nabla\varphi(f))+\frac{m}{f}d(\Delta f)-\lambda d\varphi\\
&&+\frac{1}{2}d(|\nabla\varphi|^{2})-d(\Delta\varphi).
\end{eqnarray}
Plugging the equations \eqref{Bianc1} and \eqref{Bianc2} in equation \eqref{Bianc} we have
\begin{eqnarray}\label{eqAuxProp1}
\nonumber0&=&\frac{2-n}{2}d\lambda+\frac{m}{2f^{2}}\Delta f df + \frac{m}{2f}d(\Delta f)-\frac{1}{2}d(\Delta\varphi)+\dfrac{m}{f}\lambda df+\dfrac{m(m-1)}{2f^{2}}d(|\nabla f|^{2})\\
&&-\dfrac{m}{f}d(\nabla\varphi(f))-\lambda d\varphi+\frac{1}{2}d(|\nabla\varphi|^{2}).
\end{eqnarray}
From \eqref{EQMthm2}
\begin{equation*}
\lambda d\varphi = \frac{1}{2}d\Big((2-m-n)\lambda + |\nabla\varphi|^{2} - \Delta\varphi\Big) - \frac{m}{2f}d(\nabla\varphi(f)) + \frac{m}{2f^2}\nabla\varphi(f)df.
\end{equation*}
Plugging this latter equation in \eqref{eqAuxProp1}, we obtain we obtain
\begin{eqnarray*}
0&=&\frac{m}{2}d\lambda+\frac{m}{2f^{2}}\Delta f df +\frac{m}{2f}d(\Delta f)+\dfrac{m}{f}\lambda df+\dfrac{m(m-1)}{2f^{2}}d(|\nabla f|^{2})\\
&&-\dfrac{m}{2f}d(\nabla\varphi(f)) - \frac{m}{2f^2}\nabla\varphi(f)df.
\end{eqnarray*}
Multiplying it by $\frac{2f^{2}}{m}$ we get
\begin{equation*}
0=f^2d\lambda+\Delta fdf + fd(\Delta f) + 2f\lambda df + (m-1)d|\nabla f|^2 - fd(\nabla\varphi(f)) - \nabla\varphi(f)df,
\end{equation*}
i.e.,
\begin{eqnarray*}
d\big(\lambda f^{2}+ f\Delta f +(m-1)|\nabla f|^{2}-f\nabla\varphi(f)\big)=0,
\end{eqnarray*}
which is sufficient to complete the proof.
\end{proof}

\subsection{Proof of Theorem~\ref{CARSPW}}
\begin{proof}
By the hypotheses on $f$, $\lambda$ and $\varphi$ we can conclude by Proposition \ref{PP5} that any $\mu$ given by \eqref{CMthm} is constant. Now, taking a complete Einstein manifold $(\Bbb{F}^m,g_{\Bbb{F}})$ with Ricci tensor $^{\Bbb{F}}\!Ric=\mu g_{\Bbb{F}}$, we can consider the warped product $(B^n\times_f\Bbb{F}^m, g)$ with $g=\pi^*g_B+(f\circ\pi)^2\sigma^*g_{\Bbb{F}}$. Notice that this manifold has a structure of Ricci almost soliton. Indeed, it follows from $H^\varphi(Y,Z)=\nabla^2\tilde\varphi(Y,Z)$, $H^f(Y,Z)=\nabla^2\tilde f(Y,Z)$, Eq.~\eqref{i,Cor43,O'Neill} and the hypothesis \eqref{EQMthm} that the fundamental equation
\begin{equation*}
Ric+\nabla^2\tilde\varphi=\tilde\lambda g.
\end{equation*}
is satisfied for all $Y,Z\in\mathfrak{L}(B)$.

For $Y\in\mathfrak{L}(B)$ and $V\in\mathfrak{L}(\Bbb{F})$, by Proposition \ref{Prop2}, $\nabla\tilde\varphi\in\mathfrak{L}(B)$ implies that $\nabla^2\tilde\varphi(Y,V)=0$ and from \eqref{ii,Cor43,O'Neill} we have that $Ric(Y,V)=0$. So, the fundamental equation is again satisfied.

Finally, for $V,W\in\mathfrak{L}(\Bbb{F})$ we have
\begin{equation*}
Ric(V,W)+\nabla^2\tilde\varphi(V,W)= {}^{\Bbb{F}}\!Ric(V,W)-\Big(\frac{\Delta f}{f}+(m-1)\frac{|\nabla f|^{2}}{f^{2}}\Big)g(V,W)+g(D_{V}\nabla\tilde\varphi,W).
\end{equation*}
As $^{\Bbb{F}}\!Ric=\mu g_{\Bbb{F}}$, we have from \eqref{ii,Pro35,O'Neill} that
\begin{equation*}
Ric(V,W)+\nabla^2\tilde\varphi(V,W)=\mu g_{\Bbb{F}}(V,W)-\Big(\frac{\Delta f}{f}+(m-1)\frac{|\nabla f|^{2}}{f^{2}}-\frac{\nabla\tilde\varphi(f)}{f}\Big)g(V,W).
\end{equation*}
Plugging \eqref{CMthm} in the previous equation we obtain
\begin{eqnarray*}
Ric(V,W)+\nabla^2\tilde\varphi(V,W)&=&\big(\lambda f^{2}+f\Delta f+(m-1)|\nabla f|^{2}-f\nabla\tilde\varphi(f)\big)\frac{1}{f^{2}}g(V,W)\\
&&-\Big(\frac{\Delta f}{f}+(m-1)\frac{|\nabla f|^{2}}{f^{2}}\Big)g(V,W)+\frac{\nabla\tilde\varphi(f)}{f}g(V,W)\\
&=&\tilde\lambda g(V,W)
\end{eqnarray*}
here we use that $\tilde\lambda(p,q)=\lambda(p)$ for all $(p,q)\in B\times_f\Bbb{F}$. This completes the proof of the theorem.
\end{proof}

Take a function $\varphi(t)$ such that $f(t)=\varphi'(t)>0$ on $\Bbb{R}$. By \eqref{EQMthm} we obtain
$\lambda=f'-\frac{m}{f}f''$ and by a straightforward computation, equation \eqref{EQMthm2} is equivalent to
\begin{equation*}
\frac{1}{f}f'f'' - f'''= 0
\end{equation*}
Moreover, equation \eqref{CMthm} yields
\begin{equation*}
\mu= -(m-1)\big(ff''-(f')^2\big)
\end{equation*}

In this setting, as mentioned in the introduction, we immediately recover the following example:

\begin{example}[Pigola et al.~\cite{prrs}]\label{ExRigoli}
Let $\Bbb{H}^m$ be the standard hyperbolic space. Consider the smooth functions $\varphi(t)=\sinh(t)$, $f(t)=\cosh(t)$ and $\lambda(t)=\sinh(t)-m$ all of them on $\Bbb{R}$. Then, $\big(\Bbb{R}\times_f\Bbb{H}^m,g,\nabla\tilde\varphi,\tilde\lambda\big)$ is a gradient Ricci almost soliton warped product, with $\mu=-(m-1)$.
\end{example}

\subsection{Proof of Corollary~\ref{CorCARSWP}}
\begin{proof}
For $F(\xi)=e^{-\xi}$ and $f(\xi)=e^{\xi}$ we have from \eqref{EQMthmInv1} that
\begin{equation*}
\varphi''-2\varphi'=2-m-n
\end{equation*}
whose solution is
\begin{equation*}
\varphi=\frac{c_{1}}{2}e^{2\xi}-\dfrac{(2-m-n)}{2}\xi+c_{2},
\end{equation*}
for some constants $c_1$ and $c_2$. Now, from \eqref{EQMthmInv2}, we have
\begin{equation*}
\lambda=\frac{(2-m-n)}{2}e^{-2\xi} + c_1
\end{equation*}
Now, we need to check that $f,\varphi,\lambda$ and $F$ satisfy  equation \eqref{EQMthm2}. In fact, a straightforward computation shows that
\begin{equation}\label{item(ii)}
2\lambda d\varphi=-\left(\dfrac{(2-m-n)^2}{2}e^{-2\xi}-2c_1^2e^{2\xi}\right)d\xi
\end{equation}
and
\begin{equation}\label{item(i)}
(2-m-n)d\lambda=-(2-m-n)^2e^{-2\xi}d\xi.
\end{equation}
Since $|\nabla\varphi|^2=F^2(\varphi')^2$ we deduce that
\begin{equation}\label{item(iii)}
d(|\nabla\varphi|^2)=\left(2c_1^2e^{2\xi}-\dfrac{(2-m-n)^2}{2}e^{-2\xi}\right)d\xi.
\end{equation}
We note that $\Delta\varphi=F^2\varphi''-(n-2)FF'\varphi'$, thus
\begin{eqnarray}\label{item(iv)}
\nonumber d(\Delta\varphi)&=&d\left(e^{-2\xi}2c_1e^{2\xi}+(n-2)e^{-2\xi}\left(c_1e^{2\xi}-\dfrac{(2-m-n)}{2}\right)\right)\\
&=&-(2-n)(2-m-n)e^{-2\xi}d\xi.
\end{eqnarray}
As $\nabla\varphi(f)=F^2\varphi'f'$ we have
\begin{equation}\label{item(v)}
d\left(\dfrac{m}{f}\nabla\varphi(f)\right)=m(2-m-n)e^{-2\xi}d\xi.
\end{equation}
Combining equations \eqref{item(i)}-\eqref{item(v)} we have that \eqref{EQMthm2} is satisfied. Finally, we compute
\begin{eqnarray*}
\mu &=&\lambda f^2+f\Delta f+(m-1)|\nabla f|^2-f\nabla\varphi(f)\\
&=&\left(\dfrac{(2-m-n)}{2}e^{-2\xi}+c_1\right)e^{2\xi}+e^{\xi}(e^{-2\xi}e^\xi+(n-2)e^{-2\xi}e^\xi)\\
&&+(m-1)e^{-2\xi}e^{2\xi}-e^\xi e^{-2\xi}e^\xi\left(c_1e^{2\xi}-\dfrac{(2-m-n)}{2}\right).
\end{eqnarray*}
Whence, we obtain
\begin{eqnarray*}
\mu =\dfrac{(2-m-n)}{2}+c_1e^{2\xi}+1+(n-2)+(m-1)-c_1e^{2\xi}+\dfrac{(2-m-n)}{2}=0.
\end{eqnarray*}
The conclusion of the corollary now immediately follows from Theorem~\ref{CARSPW}.
\end{proof}

\subsection{Proof of Theorem~\ref{thmTrivial}}
\begin{proof}
Consider the elliptic operator of second order given by
\begin{equation}
\mathcal{E}(\cdot):=\Delta(\cdot)-\nabla\varphi(\cdot)+\frac{m-1}{f}\nabla f(\cdot).
\end{equation}
By Proposition~\ref{PP2}, equation \eqref{EQMthm} (of Proposition~\ref{PP5}) is valid and $^\Bbb{F}\!Ric=\mu g_{\Bbb{F}}$ with
\begin{equation}\label{eqAux1Thm1}
\mu=\lambda f^{2}+f\Delta f+(m-1)|\nabla f|^{2}-f\nabla\varphi(f).
\end{equation}
or equivalently
\begin{equation*}
\mathcal{E}(f)=\Delta f-\nabla\varphi(f)+\frac{(m-1)}{f}|\nabla f|^2=\frac{\mu-\lambda f^2}{f}.
\end{equation*}
Moreover, equation~\eqref{EQMthm2} (also of Proposition~\ref{PP5}) is valid by Proposition~\ref{Prop2}. Hence, we are in the hypothesis of Proposition~\ref{PP5} which ensures that $\mu$ is constant. Then the result of item~\eqref{thmTrivial-item1} follows from the strong maximum principle.

For item \eqref{thmTrivial-item2}, let $p,q\in B^n$ be the points where $f$ reaches its maximum and minimum in $B^n$, respectively. Then
\begin{equation*}
\nabla f(p)=0=\nabla f(q) \quad \mbox{and} \quad \Delta f(p)\leq 0\leq \Delta f(q).
\end{equation*}
Since $f>0$ and $\lambda(p)\leq\lambda(q)$ we have $-\lambda(p) f(p)^{2}\geq-\lambda(q)f(q)^{2}$ and combining this with \eqref{eqAux1Thm1} we get
\begin{equation*}
0\geq f(p)\Delta f(p)=\mu-\lambda(p)f(p)^{2}\geq \mu-\lambda(q) f(q)^{2}=f(q)\Delta f(q)\geq 0.
\end{equation*}
Consequently,
\begin{equation}\label{Ric13}
\mu-\lambda(p)f(p)^{2}=\mu-\lambda(q)f(q)^{2}=0
\end{equation}
Consider initially the case where $\lambda(p)\neq 0$, then the last equation implies $\lambda(q)\neq 0$ and as $\lambda(p)\leq\lambda(q)<0$, we get
\begin{equation*}
f(p)^{2}=\Big(\dfrac{\lambda(q)}{\lambda(p)}\Big)f(q)^{2}\leq f(q)^{2}.
\end{equation*}
Thus, $f(p)=f(q)$, i.e., $f$ is constant. So is $\lambda$ by \eqref{eqAux1Thm1}.

Now, if $\lambda(p)=0$, by \eqref{Ric13}, $\lambda(q)=0$ and $\mu=0$. This way, equation \eqref{eqAux1Thm1} implies that
\begin{equation*}
\mathcal{E}(f)=\Delta f-\nabla\varphi(f)+\frac{(m-1)}{f}|\nabla f|^{2}=-\lambda f\geq 0.
\end{equation*}
Hence, by the strong maximum principle $f$ is constant and $\lambda$ is null by \eqref{eqAux1Thm1}. In either case $M$ is simply a Riemannian product.
\end{proof}

\section{Acknowledgements}
The authors would like to express their sincere thanks to the referee for careful reading and useful comments which improved this paper. They are also grateful to Dragomir Tsonev for useful comments, discussions and constant encouragement. Jos\'e N.V. Gomes would like to express their sincere thanks to the Department of Mathematics-Lehigh University, where part of this work was carried out. He is grateful to Huai-Dong Cao and Mary Ann for their warm hospitality and constant encouragement. This work has been partially supported by CAPES-Brazil, FAPEAM-Brazil, CNPq-Brazil and FAPEG-Brazil.

\end{document}